\documentclass[11pt, a4paper]{article}
\usepackage{amsfonts}
\usepackage{color}
\usepackage{mathrsfs}
\usepackage{url}
\usepackage{amssymb}
\usepackage[latin1]{inputenc}
\usepackage[T1]{fontenc}
\usepackage[english]{babel}
\usepackage{listings}
\usepackage{graphicx, subfigure}
\usepackage{epsfig,amstext,amsthm,latexsym}
\usepackage{epstopdf}
\usepackage{relsize}
\usepackage{float}
\usepackage{epsfig}
\usepackage{fancyhdr}

\usepackage[centertags]{amsmath}
\usepackage{graphicx}
\usepackage{multirow}
\usepackage{array}
\usepackage[figuresright]{rotating}
\usepackage{caption}
\usepackage{mathrsfs}
\usepackage{url}
\usepackage{amsmath}
\usepackage[latin1]{inputenc}
\usepackage[T1]{fontenc}
\usepackage[english]{babel}
\usepackage{listings}
\usepackage[authoryear]{natbib}
\bibpunct{(}{)}{;}{a}{,}{,}
\allowdisplaybreaks[4]
\numberwithin{equation}{section}
\allowdisplaybreaks[4]
\usepackage{graphicx, subfigure}
\usepackage{tabularx,ragged2e,booktabs,caption}
\definecolor{c20}{rgb}{0.,1,0.}
\definecolor{c30}{rgb}{0.,0.,1.}
\definecolor{c40}{rgb}{1,0.1,0.7}
\definecolor{c50}{rgb}{1,0,0}
\definecolor{c60}{rgb}{1,0.9,0.1}

\def\cL#1{\textcolor{c50}{#1}}
\def\cL#1{#1}

\def\cg#1{\textcolor{c50}{#1}}
\def\cg#1{#1}

\def\hat{\widehat}
\def\bar{\overline}

\def\IF{\infty}
\begin{document}
\newcommand{\BQN}{\begin{eqnarray}}
\newcommand{\EQN}{\end{eqnarray}}
\newcommand{\BQNY}{\begin{eqnarray*}}
\newcommand{\EQNY}{\end{eqnarray*}}
\newcommand{\nelem}[1]{{Lemma \ref{#1}}}
\newcommand{\neprop}[1]{{Proposition \ref{#1}}}
\newcommand{\netheo}[1]{{Theorem \ref{#1}}}
\newcommand{\prooftheo}[1]{ \textsc{Proof of Theorem} \ref{#1} }
\newcommand{\proofprop}[1]{\textsc{Proof of Proposition} \ref{#1}}
\newcommand{\prooflem}[1]{\textsc{Proof of Lemma} \ref{#1}}
\newcommand{\proofkorr}[1]{\textsc{Proof of Corollary} \ref{#1}}

%ÐÂµÄÃüÁî»·¾³
\newtheorem{theorem}{\bf Theorem\,}[section]%É¾³ý[section],Ôò±àºÅ²»·Ö½Ú
\newtheorem{lemma}{\bf Lemma\,}[section]
\newtheorem{remark}{\bf Remark\,}[section]
\newtheorem{example}{\bf Àý\,}[section]
\newtheorem{proposition}{\bf ÃüÌâ\,}[section]
\newtheorem{corollary}{\bf ÍÆÂÛ\,}[section]
\newtheorem{definition}{\bf Definition \,}[section]
\renewcommand\refname{\large References}
\renewcommand\arraystretch{0.7}
%ÕýÎÄ

\title{Robust estimations for the tail index of Weibull-type distribution\footnote{Financial support from the National Natural Science Foundation of China grant (11604375), and Chinese Government Scholarship (201708505031), are gratefully acknowledged.}}
%\subtitle{}

\author{Chengping Gong\footnote{School of Mathematics and Statistics, Southwest University, 400715 Chongqing, China.
E-Mail: chengping.gong@foxmail.com}, Chengxiu Ling\footnote{School of Mathematics and Statistics, Southwest University, 400715 Chongqing, China and Department of Actuarial Science, University of Lausanne, Chamberonne 1015 Lausanne, Switzerland, E-Mail: chengxiu.ling@unil.ch}}
%\date{\svnday.\svnmonth.\svnyear, SVN-revision: \svnrev}
\date{\;}

\maketitle
\linespread{1.66}
\setlength\parindent{0.15in}
\setlength\parskip{0.15in}

\newcommand{\quantnet}{\hspace*{\fill} \raisebox{-1pt}{\includegraphics[scale=0.05]{qletlogo}}\,}

\begin{abstract}
%\small{\textbf{
%\begin{center}
%Abstract
%\end{center}
%}}
\footnotesize{
Based on suitable left-truncated or censored data,
two flexible classes of $M$-estimations of Weibull tail coefficient are proposed with two additional parameters bounding the impact of extreme contamination. Asymptotic normality with $\sqrt n$-rate of convergence is obtained. Its robustness is discussed via its asymptotic relative efficiency and influence function.  It is further demonstrated by a small scale of simulations and an empirical study on CRIX.
\noindent

\noindent {\bf Keywords}: robust; Weibull tail coefficient; influence function; asymptotic relative efficiency; CRIX}

\noindent {\bf JEL Classification}: C13, G10, G31

\noindent {\bf Mathematics Subject Classification (2010)}: 60G70, 62G32, 91B28

\end{abstract}

\newpage

%========================================================================================================
\section{Introduction}\label{Introduction}

\cg{The} estimation of tail quantities plays an important role in extreme value statistics. One challenging problem is to select {extreme} sample fraction to balance the asymptotic variance and bias. Meanwhile, this requires a large and ideal sample from the underlying distribution. Indeed, in practical data analysis, it is not unusual to encounter outliers or mis-specifications of the underlying model which may have a considerable impact on the estimation results.~A typical treatment is then required for instance by down-weighting  its influence on the estimation in various standards, see~e.g., \cite{basu1998robust}, \cite{beran2012on}, \cite{vandewalle2004robust,vandewalle2007robust}, \cite{goegebeur2015robust,Liu2010subexponential}.

{Given the wide applications of Weibull-type distributions and little studies on its robust estimations}, this paper shall address this issue concerning its tail quantities. Let  $X_1, \ldots, X_n$ be an~independent and identically distributed sequence from parent $X\sim F(x)$ satisfying
\BQN
\label{def_Weibull}
1- F (x)=\exp\{-x^{\alpha}\ell(x)\} \quad \mbox{for large } x,
\EQN
where $\alpha >0$ is the so-called {Weibull tail coefficient} (WTC)  and $\ell(x)$ is a slowly varying function at infinity, i.e., (cf. \cite{bingham1987regular})
\[
\lim_{t \rightarrow \infty} \ell(tx) / \ell(x)=1, \quad \forall x>0.
\]

Prominent instances of Weibull-type distributions of $F$ are  Gaussian ($\alpha = 2$), gamma, Logistic and \cg{exponential} ($\alpha=1$) and \cg{extended} Weibull (any $\alpha>0$) distributions (cf. \mbox{\cite{gardes2008estimation}}). {As} an important subgroup of light-tailed distributions, Weibull-type distributions are of great use in hydrology, meteorology, environmental and actuarial science, to name but a few (cf. \cite{beirlant1992modeling, Hashorva2014tail, Debicki2018extremes, Arendarczyk2011asymptotics}). Meanwhile, the WTC governs the tail behavior of $F$, {and the larger the WTC is, the faster the tail of $F$ decays.} Dedicated estimations of WTC have {thus} been proposed and most of them are based on an asymptotically vanishing sample fraction of high {quantiles, which asymptotic} normality is achieved under certain second-order condition specifying the rate of convergence of $\ell(tx)/\ell(t)$ to 1, see e.g., \cite{girard2004hill}, \cite{gardes2008estimation}, \cite{goegebeur2010generalized}, \cite{Asimit2010pitfall}.
{Indeed, most data-sets from {applied-oriented fields} are relative large and with certain deviations from the pre-supposed model. For instance, it occurs with} the slowly varying function $\ell(\cdot)$ (where $1- F (x)=\exp\{-x^{\alpha}\ell(x)\}$) in the left part of the distributions. To the best of our knowledge, it is new to  investigate the robust Weibull tail  estimations when only a small sample is available.

Inspired by the theory of robust inference in \cite{Huber1964}, we propose two  classes of robust estimations of WTC. % Following similar idea in \cite{beran2012on}, we start with the $\psi$-function as below.
{Denote for given $c_0>0$}
\BQN
\label{def_g}
h(t)= (c_0t -1)\ln t -1, \quad t>0.
\EQN
Clearly, we have $g(x; \alpha)= -\alpha^{-1} h(x^\alpha)$ is the score function of
\BQN
\label{def_F}
X\sim F_W(x; \alpha) = 1- \exp\{-c_0 x^\alpha\}, \quad x>0, \, \alpha>0.
\EQN
{Please note that $h(t),\,t>0$ is not monotone and thus one cannot directly weaken the effect of outliers by bounding score function $g(x; \alpha)$.} On the other hand, most interest of risk management lies principally in the extreme large risks. This motivates us to consider some tailored  $h(t)$ according to certain left-truncated/censored Weibull distributions with the same Weibull tail coefficient $\alpha$ under considerations. Namely, we set below $t_0 =\arg \min_{t\ge 1}{h(t)}$ with $h$  specified by \eqref{def_g} and
\BQN
\label{def_h}
\tilde h(t) = h(t),\, t\ge1,\quad h^*(t) = h(t),\, t\ge t_0,
\EQN
which properties are stated as below.

\begin{lemma}
\label{L1}
{Let $X\sim F_W(x; \alpha)$ and $\tilde X= X |\{ X \ge 1\} \sim \tilde F_W(x; \alpha)$.  Then $\tilde h^{\leftarrow}(y)= \inf\{t \ge 1: \tilde h(t) \ge y\}, y\ge -1$} is strictly increasing, and $-\alpha^{-1} \tilde h(x^\alpha)$ is the score function of $\tilde F_W(x; \alpha)$. Moreover, $h^*(x^\alpha) , \, x\ge x_0 = t_0^{1/d_0}$
is strictly increasing provided that $\alpha \ge d_0>0$.
\end{lemma}
Basically, both $\tilde h^\leftarrow$ and $h^*$ are certain modifications of $h$ via its {valued interval and domain region. Now, we are ready to state} our $M$-estimations of Weibull tail coefficient using {the $M$-estimation process based on the alternative samples}  $\tilde X_i$'s and $X_i^*$'s  respectively from  \mbox{$\tilde X:= X |\{ X \ge 1\} \sim \tilde F$} and $X^* := \max(X, x_0)\sim F^*$ {where $X_i$'s is a random sample from $X\sim F$}. Set below $[y]_{v}^{u}=\min(\max(y, v), u),\}, v<u$ and $\Im$ is a set of distributions with support in $(0, \infty)$.

\begin{definition}
\label{def_Tt} Let $F_W(x; \alpha)$ and \cg{$\tilde h, h^*$} be given by \eqref{def_F} and \eqref{def_h}, respectively.
Define the psi-function $\tilde\psi$ as
\BQN
\label{def_psi1}
\tilde \psi_{v, u}(y; \alpha) &=& [\tilde h(y^\alpha)]_v^u - \int_{1}^\infty  [\tilde h(z^\alpha)]_v^u d \tilde F_W(z; \alpha) \notag\\
&=& [\tilde h(y^\alpha)]_v^u - \left[v + \int_v^u \exp\{-c_0 [\tilde h^\leftarrow(z)-1]\} \, dz\right], \, -1 \le v< u < \infty.
\EQN
\cg{Then the functional $\tilde T(F)$ as the solution of the equation}
\BQNY
\tilde \lambda_F(t)=\int_{1}^\infty \tilde\psi_{v, u}(y; t)d \tilde F(y) =0,  \quad F\in\Im,
\EQNY

is called huberized {Weibull tail} $M$-functional corresponding to $\tilde \psi$. The corresponding $M$-estimator $\tilde T_{n}= \tilde T^{(v, u)}_n(F_n)$,  the solution of the equation
\BQNY
\tilde \lambda_{F_n}(t) =\sum_{j=1}^{m} \tilde \psi_{v,u}(\tilde X_{j}; t)=0,\quad m = \#\{1\le i\le n: X_i \ge 1\},
\EQNY
is the  huberized Weibull tail $M$-estimator of $\alpha$. If further \cg{$0 < d_0 \leq \alpha\leq d_1$}, then define the psi-function $\psi^*$ with \cg{$x_0 = \big(\arg \min_{t\ge 1}{h(t)}\big)^{1/d_0}$ and $ v_0 = h(x_0^{d_1})$}
\BQN
\label{def_psi}
\psi_{v, u}^*(y; \alpha) &=& [h^*(y^\alpha)]_v^u - \int_{x_0}^\infty  [h^*(z^\alpha)]_v^u d F_W^*(z; \alpha) \notag\\
&=& [h^*(y^\alpha)]_v^u - \left[v + \int_v^u \exp\{-c_0(h^*)^\leftarrow(z)\} \, dz\right], \quad v_0\le v< u<\infty\cL.
\EQN
Then the functional $T^*(F), F\in \Im$ as the solution of the equation
\BQNY
\lambda^*_{F}(t)=\int_{x_0}^\infty  \psi^*_{v, u}(y; t)d F^*(y)=0,
\EQNY

is called huberized {Weibull tail}  $M$-functional corresponding to $\psi^*$.~The corresponding $M$-estimator $T^*_{n}= T_n^{*(v, u)}(F_n)$,  the  solution of the equation
\BQNY
\lambda^*_{F_n}(t) = \sum_{i=1}^{n}\psi _{v,u}(X_{i}^*; t)=0,
\EQNY
is the huberized $M$-estimator of the Weibull tail coefficient $\alpha$.
\end{definition}
We remark that \eqref{def_psi1} and \eqref{def_psi} \cg{hold} since
\BQNY
\tilde F_W (x; \alpha) =1-\exp\{-c_0[x^\alpha -1]\}, \quad x\ge1; \quad F^*_W(x; \alpha) =  \left\{\begin{array} {ll}
	1-\exp\{-c_0 x^\alpha\}, &x> x_0,\\
	0, & x \le x_0.
\end{array}
\right.
\EQNY

	\begin{figure}[!htb]
	\centering
	\setlength{\abovecaptionskip}{0pt}
	\includegraphics[width=16cm,height=8cm]{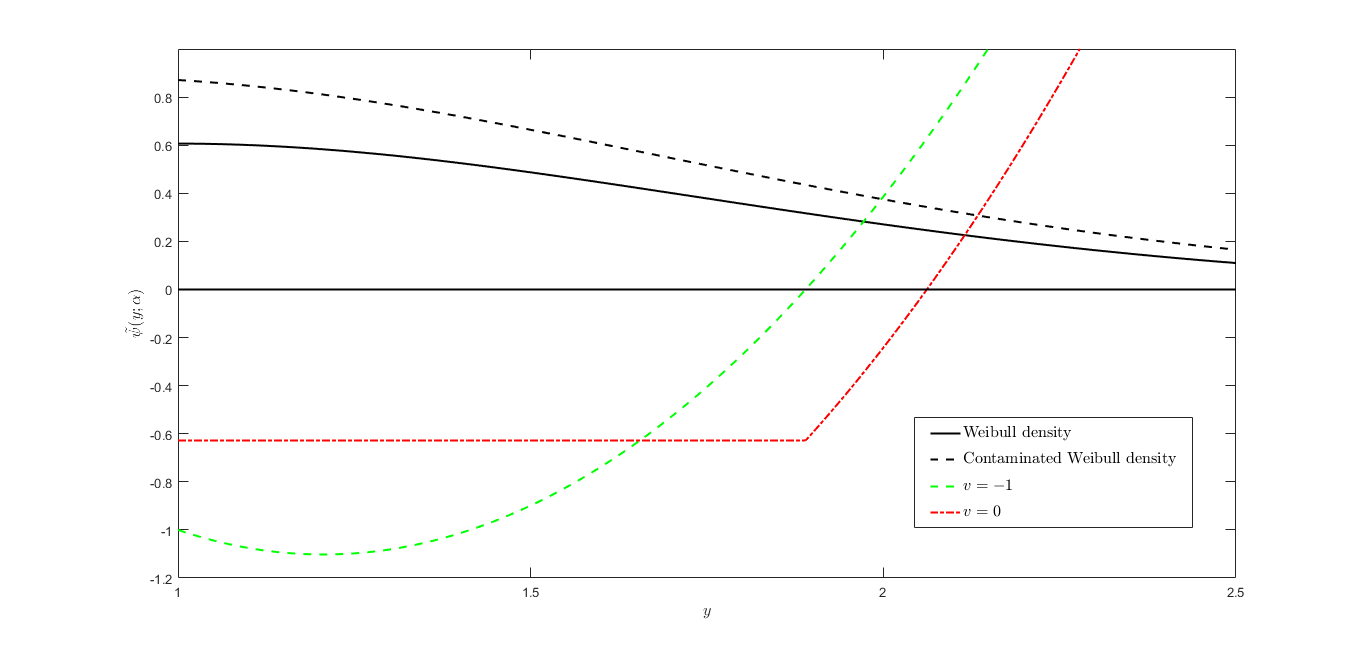}\\
	\caption{ \cg{Psi-functions $\tilde \psi$} for the huberized $M$-estimators $\tilde{T}^{(v, \infty)}_n$. Here the truncated densify functions are generated from the Weibull $F_W(x; \alpha)$ and contaminated Weibull $F_\epsilon(x) = (1-\epsilon) F_W(x; \alpha) + \epsilon \Gamma(x; \lambda, \beta)$ with $\alpha=2, c_0=0.5, \epsilon=0.3, \lambda= 1, \beta=1$.}\label{fig.1}
\end{figure}
{Figure \ref{fig.1} illustrates the lower huberization by comparing the score function
	$\tilde\psi_{-1,\IF}(y; \alpha)$ of $\tilde F_W$ (recall \nelem{L1}) with $\tilde\psi_{v,\IF}(y; \alpha)$.  We see that the {contaminated Weibull density by Gamma (see \eqref{d Gamma} below for its definition)} has almost the same shape as the \cg{pre-supposed} Weibull one in the right tail, and therefore lower-huberized psi-function $\tilde{\psi}_{v,\infty}(y;\alpha)$ {can} restrict the influence of all observations below $y_0 = (\tilde h^\leftarrow(v))^{1/\alpha}$ instead of removing them completely. On the other hand, for all $y > y_0$, the $\tilde\psi_{v, \IF}(y; \alpha)$ is shifted downwards for the consistency purpose. One may similarly analyze the $\psi^*$ function.
	
	The paper principally investigates the asymptotic behavior of the proposed new classes of $M$-estimations of Weibull tail coefficient. Details are as follows.
	
	In Section \ref{Section 2}, we consider Weibull distributions in Theorems \ref{T1} and \ref{T2} {and establish} its asymptotic normality of the $M$-estimations $\tilde T_n$ and $T_n^*$ {with} $\sqrt n$-rate of convergence, which is rather faster than that of most classical Weibull tail estimations such as the Hill-type estimation, see Theorem 2 in \mbox{\cite{girard2004hill}}.~Generally, we study related asymptotic properties in Theorems \ref{T3} and \ref{T4} {when the underlying risk follows Weibull-type distributions specified in \eqref{def_Weibull}. Some bounded asymptotic bias may appear due to its deviations}  from the Weibull distributions.
	
	In Section \ref{Section 3}, using asymptotically relative efficiency (AEFF) and influence function (IF), we~investigate the robustness (\netheo{T5}) and the bias, which \cg{are} further related to the \cg{choices} of flexible parameters $v$ and $u$. These results {are useful, especially when the practical regulators in risk management consider} the trade-off between the robustness and consistency.
	
	In Section \ref{Section 4}, a small scale of Monte Carlo simulations and an empirical study concerning the CRIX proposed by \cite{trimborn2016CRIX} are carried out. We see that {both} $M$-estimations {are robust} and perform very well even for small samples, in comparisons with the classical  maximum likelihood estimations and Hill-type  estimations of the Weibull tail coefficient. We expect the results would be beneficial to both financial practitioners and theoretical experts \cg{in} risk management and extreme value statistics.
	
	The rest of the paper is organized as follows. Main results are given in Section \ref{Section 2} followed with a~section dedicated to the robust analysis. Sections \ref{Section 4} and \ref{Empirical} are devoted to a small scale of Monte Carlo simulations {and an empirical studies on CRIX}.
	All proofs of the results are postulated to  Section \ref{Section 5}.

%======================================================================================================
\section{Asymptotic results}\label{Section 2}
Throughout this section, we keep the same notation as in Introduction and write further $\stackrel{p}{\to}$ and $\stackrel{d}{\to}$ for the convergence in probability and in distribution, respectively. All the limits are taken as $n \to \infty$ unless otherwise stated.

\begin{theorem}\label{T1}
{Let $X_{1},\ldots, X_{n}$ be a random sample from $X\sim F_W(x;\alpha_0)=1-\exp\{-c_0x^{\alpha_0}\},\, x>0, c_0, \alpha_0 >0$. % and $d_0 \le \alpha_0\le d_1$ for some known positive constants $c_0, d_0< d_1$.
	Denote by ${\tilde X_j} \sim \tilde X = X| \{X \ge1\}, \, 1\le j \le {m=\#\{1\le i \le n: X_i \ge 1\}}$, and \cg{by} $\tilde T_n = \tilde T_n^{(v, u)}, -1 \le v < u \cg{<} \infty$} the solution of
\BQNY
\tilde \lambda_{(F_{W})_{n}}(t)=\sum_{i=1}^{m} \tilde\psi _{v,u}(\tilde X_i; t)=0.
\EQNY
Then $\tilde T_n  \stackrel{p}{\to} \alpha_0$ and
\BQN
\label{conv_d}
\sqrt n(\tilde T_n-\alpha_0)\stackrel{d}{\to} N(0,\tilde\sigma^2_{v,u,\alpha_0;F_W}),
\EQN
where, with $\tilde\mu = v + \int_v^u \exp\{-c_0[\tilde h^\leftarrow(z)-1] \} dz$
$$\tilde\sigma^2_{v,u,\alpha_0;F_W}= e^{c_0} \frac{\alpha_0^2}{c_0^2}\frac{(v-\tilde \mu)^{2}+ 2\int_{\tilde h^\leftarrow(v)}^{\tilde h^\leftarrow(u)} [\tilde h(s) - \tilde\mu]\exp\{- c_0(s-1)\} d \tilde h(s)}{\left[\int_{\tilde h^\leftarrow (v)}^{\tilde h^\leftarrow(u)} s \ln  s \exp\{-c_0(s-1) \}  d \tilde h(s)\right]^2}.$$
\end{theorem}

{\remark
\label{R1}
{As stated in \nelem{L1},} $-\alpha^{-1} \tilde h(x^\alpha), x\ge1$ is the score function of $\tilde F_W(x; \alpha)$. Therefore, $\tilde T_n = \tilde T_n^{(v, u)}$ with $v=-1, u= \infty$ reduces to the maximum likelihood estimation of $\alpha$. This fact will be used in \netheo{T5} for the asymptotic relative efficiency analysis. Additionally, we have by laws of large numbers that $m= m(n)$ {satisfies} %such that
$m/n \cg{\stackrel{p}{\to}} P\{X \ge 1\}= e^{-c_0}$.
}

\begin{theorem}\label{T2}
Let $X_{1}, \cg{\ldots}, X_{n}$ be a random sample from $X\sim F_W(x;\alpha_0)=1-\exp\{-c_0x^{\alpha_0}\},\, x>0$ and \cg{$ 0 < d_0 \leq \alpha_0\leq d_1$}. Denote by $X_i^* = \max(X_i, x_0), 1\le i\le n$ with $x_0=(\arg\min_{t\ge1} h(t))^{1/d_0}$, and by $T_n^* = T_n^{*(v, u)}, \, n\in \mathbb{N},\, u> v\ge \cg{v_0= h(x_0^{d_1})}$ the solution of
\BQNY
\lambda^*_{(F_W)_{n}}(t)=\sum_{i=1}^{n}\psi^*_{v,u}(X_i^*; t)=0.
\EQNY
Then $T_n^*  \stackrel{p}{\to} \alpha_0$ and
\BQNY
\sqrt n(T_n^*-\alpha_0)\stackrel{d}{\to} N(0,\sigma^{*2}_{v,u,\alpha_0;F_W}),
\EQNY
where, with $\mu^* = v + \int_v^u \exp\{-c_0(h^*)^\leftarrow(z)\} \, dz$
$$\sigma^{*2}_{v,u,\alpha_0;F_W}=\frac{\alpha_0^2}{c_0^2} \frac{(v-\mu^*)^{2}+ 2\int_{(h^*)^\leftarrow(v)}^{(h^*)^\leftarrow(u)} [h^*(s) - \mu^*]\exp\{- c_0s\} d h^*(s)}{\left[\int_{(h^*)^\leftarrow (v)}^{(h^*)^\leftarrow(u)} s \ln  s \exp\{-c_0 s \}  d h^*(s)\right]^2}.$$
\end{theorem}
{\remark
	\label{R2}
	(i) The difference between \cg{$\tilde\psi$ and $\psi^*$} is that $h^*$ is not the score function of $F_W^*$, the distribution of the censored risk at point $x_0$, \cg{where $0 < d_0 \le \alpha_0 \le d_1$ is needed to ensure the monotonicity of $h^*$ and $\psi^*$, see details in \eqref{def_h^*} with $\alpha = \alpha_0$}. \\
	(ii) The proposed $M$-estimations are principally based on suitable left-truncated and censored data, which are \cg{commonly} used in survival analysis, see e.g., \cite{Kundu2017analysis}. Moreover, both consistency and robustness are obtained since we bound the psi-functions to weaken the influence of the extreme outliers for the exact Weibull~models.
}

In what follows, we consider generally the Weibull-{type} risks and {investigate asymptotic properties of \cg{the} proposed $M$-estimations. %Two additional conditions are then needed} to ensure the existence of the estimations and the finiteness of asymptotic variance.

\begin{theorem}\label{T3}
Let $X_{1}, \ldots, X_{n}$ be a random sample from  $F(x)=1-\exp\{- c_0x^{\alpha_0} \ell(x)\},\, x>0$. Suppose \cg{that} there is a unique solution $t_0$ of $\tilde\lambda_F(t)=0.$ Then
\cg{$\tilde T_n = \tilde T_n^{(v, u)}, -1 \leq v < u < \infty$}, the solution of
$$\tilde \lambda_{F_n}(t)=\sum_{i=1}^{m} \tilde \psi_{v,u}(\tilde X_i;t)=0, \quad m =\#\{1\le i\le n: X_i \ge 1\},
$$
converges in probability to $t_0$.
If further $\int_{1}^\infty \tilde \psi^2_{v, u}(x; t_0)\, d \tilde F(x) < \infty$ and $\tilde \lambda_F'(t)\neq 0$ \cg{hold} in a neighbourhood of $t_0$, then
\BQN
\label{conv_dl}
\sqrt n( \tilde T_n - t_0)\stackrel{d}{\to} N(0, \tilde \sigma^2_{v,u,t_0;F}),
\EQN
where
$$\tilde \sigma^2_{v,u,t_0;F} = e^{c_0 \ell(1)}\frac{\int_{1}^\infty \tilde \psi^2_{v,u}(x; t_0) d \tilde F(x)}{[\tilde \lambda_F'(t_0)]^2}.$$
\end{theorem}

\begin{theorem}\label{T4}
Let $X_{1},\ldots, X_{n}$  be a random sample from  $F(x)=1-\exp\{- c_0x^{\alpha_0} \ell(x)\},\, x>0$ and \cg{$0 < d_0 \le \alpha_0 \le d_1, x_0 = (arg min_{t \ge 1} h(t))^{1/d_0}$.}
Suppose \cg{that} there is a unique solution $t_0$ of $\lambda_F^*(t)=0.$ Then
\cg{$T_n^* = T_n^{*(v, u)}, u > v \ge v_0 = h(x_0^{d_1})$}, the solution of
$$\lambda_{F_n}^*(t)=\sum_{i=1}^{n} \psi^*_{v,u}(X_i^*;t)=0 , \quad n\in \mathbb{N}\cg{,}
$$
converges in probability to $t_0$.
If further $\int_{x_0}^\infty \psi^{*2}_{v, u}(x; t_0)\, d F^*(x) < \infty$ and $(\lambda_F^*)'(t)\neq 0$ \cg{hold} in a neighbourhood of $t_0$, then
\BQN
\label{conv_d2}
\sqrt n(T_n^* - t_0)\stackrel{d}{\to} N(0,\sigma^{*2}_{v,u,t_0;F}),
\EQN
where
$$\sigma^{*2}_{v,u,t_0;F} = \frac{\int_{x_0}^\infty\psi_{v,u}^{*2}(x; t_0) d F^*(x)}{[(\lambda_F^*)'(t_0)]^2}.$$
\end{theorem}

Please note that here the $t_0$, {the unique solution of $\tilde \lambda_F$ and $ \lambda^*_F$ specified in Theorems \ref{T3} and \ref{T4}, might not be} equal to $\alpha_0$. {In other words, to maintain the robustness of the $M$-estimations is  at cost of consistency. In the next section, we shall discuss the balance via the flexible parameters $v$ and}
$u$.

\section{Robustness}\label{Section 3}

A simple criterion for choosing $v$ and $u$ in the $M$-estimations is  the trade-off between the efficiency loss (that one is willing to put up with when data are generated by a Weibull distribution), and~its asymptotic bias (when the underlying distribution deviates from the ideal Weibull distribution).
We~study below the  relative asymptotic efficiency (AEFF) in \netheo{T5}, and then analyze its influence function. {Both quantities are some functions of the flexible parameters $v$ and $u$, which enable the risk regulators to balance the robustness and consistency.}

As stated in Remark \ref{R1}, the $M$-estimation $\tilde T_n^{(v, u)}$ with $v=-1, u=\infty$ reduces to the maximum likelihood estimation of $\alpha$. Therefore, a  straightforward application of Theorems \ref{T1} and \ref{T2} leads to the following theorem.

\begin{theorem}
\label{T5}
Under the same \cg{assumptions} of Theorems \ref{T1} and \ref{T2},
\cg{we have} the relative asymptotic efficiency \cg{functions of $\tilde T_n^{(v,u)}$ and $T_n^{*(v, u)}$} (compared to $\tilde T_n^{(-1, \infty)}$, the maximum likelihood estimation) \cg{are} given by
\BQNY
&&\quad AEFF(\tilde T_n^{(v,u)}) = \frac{\tilde \sigma^2_{-1,\infty, \alpha_0; F_W}}{\tilde \sigma^{2}_{v,u,\alpha_0; F_W}} \\
&&= \frac{\left[\int_{-1}^\infty \exp\{-c_0[\tilde h^\leftarrow(z)-1]\} \, dz\right]^{2}+ 2\int_{1}^{\infty} [\tilde h(s) - \tilde\mu]\exp\{- c_0(s-1)\} d \tilde h(s)}{\left[\int_{1}^{\infty} s \ln  s \exp\{-c_0(s-1) \}  d \tilde h(s)\right]^2} \\
&&\quad\times \frac{\left[\int_{\tilde h^\leftarrow(v)}^{\tilde h^\leftarrow(u)} s \ln  s \exp\{-c_0(s-1) \}  d \tilde h(s)\right]^2}{(v-\tilde \mu)^{2}+ 2\int_{\tilde h^\leftarrow(v)}^{\tilde h^\leftarrow(u)} [\tilde h(s) - \tilde\mu]\exp\{- c_0(s-1)\} d \tilde h(s)}
\EQNY
and
\BQNY
&&\quad AEFF(T_n^{*(v,u)}) = \frac{\tilde \sigma^2_{-1,\infty, \alpha_0; F_W}}{\sigma^{*2}_{v,u,\alpha_0; F_W}} \\
&&= e^{c_0}\frac{\left[\int_{-1}^\infty \exp\{-c_0[\tilde h^\leftarrow(z)-1]\} \, dz\right]^{2}+ 2\int_{1}^{\infty} [\tilde h(s) - \tilde\mu]\exp\{- c_0(s-1)\} d \tilde h(s)}{\left[\int_{1}^{\infty} s \ln  s \exp\{-c_0(s-1) \}  d \tilde h(s)\right]^2} \\
&&\quad\times \frac{\left[\int_{(h^*)^\leftarrow (v)}^{(h^*)^\leftarrow(u)} s \ln  s \exp\{-c_0 s \}  d h^*(s)\right]^2}{(v-\mu^*)^{2}+ 2\int_{(h^*)^\leftarrow(v)}^{(h^*)^\leftarrow(u)} [h^*(s) - \mu^*]\exp\{- c_0s\} d h^*(s)}.
\EQNY

Here $\tilde\mu$ and $\mu^*$ are given by Theorems \ref{T1} and \ref{T2}, respectively.
\end{theorem}

Figure \ref{fig.2} illustrates the effect of $v$ on the relative asymptotic efficiency of $\tilde{T}^{(v, \infty)}_n$ and $T^{*(v, \infty)}_n$ (compared to the MLE $\tilde T_n^{(-1,\IF)}${)}. {For smaller $v$,  the relative asymptotic effective loss of $\tilde T_n$ is rather smaller than that of $T_n^*$. \cg{While} for larger $v$, both are asymptotically the same.}

The influence function approach, known also as the ``infinitesimal approach'', is generally employed to quantify robustness. Recall that the influence function describes the effect of some functional $T(F)$ for $F$ in an infinitesimal $\epsilon$-contamination neighbourhood $\{F_\epsilon | F_\epsilon(x) = (1-\epsilon)F(x)+\epsilon G(x)\}$}, is defined by

\begin{figure}[H]
	\centering
	\setlength{\abovecaptionskip}{0pt}
	\includegraphics[width=16cm,height=8cm]{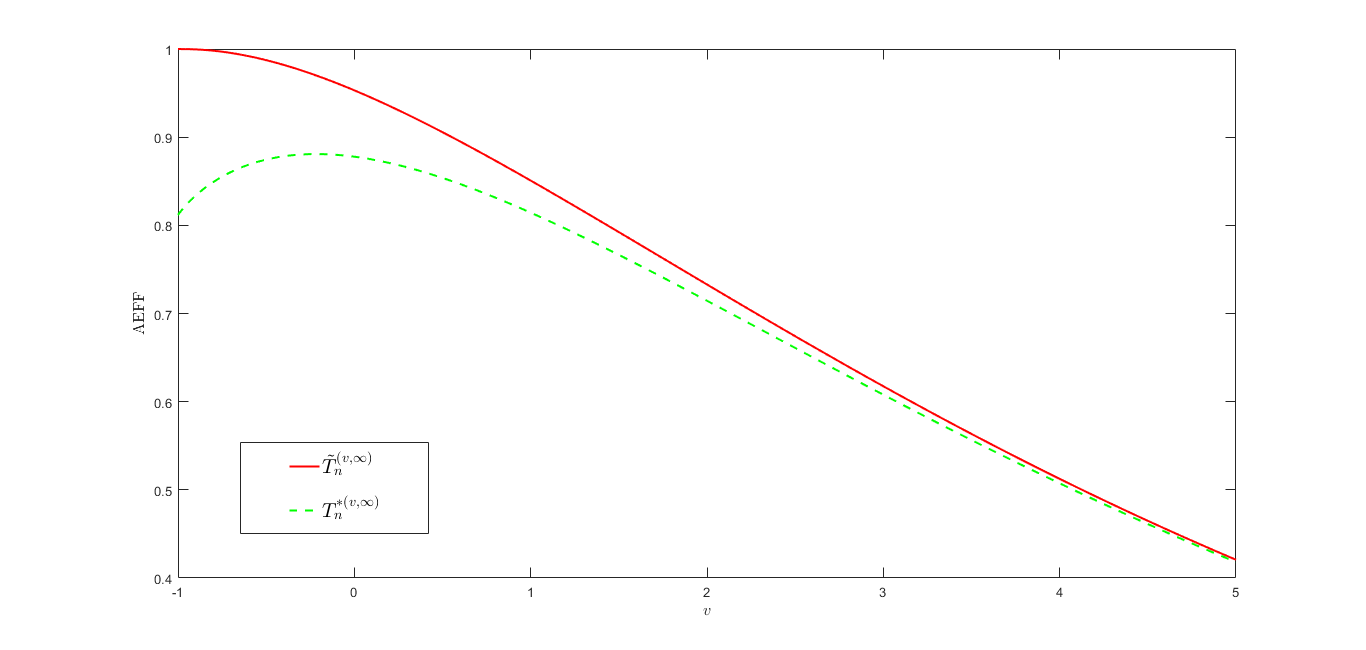}\\
	\caption{{Relative asymptotic efficiency (AEFF) of $\tilde T_n^{(v,\IF)}$  and $T_n^{*(v,\IF)}$ compared to the MLE $\tilde{T}^{(-1, \infty)}_n$. Here $F_W(x; {\alpha_0})$ is given by \eqref{def_Weibull} with $c_0=1, \alpha_0=1$}.}\label{fig.2}
\end{figure}

We have
\BQNY
IF(\tilde T;F,G) = -\frac{\int_1^\infty \tilde \psi_{v,u}(y;\tilde T(F)) d \tilde G(y)}{\tilde \lambda_F'(\cg{\tilde T(F)})}; \quad  IF(T^*;F,G) = -\frac{\int_{x_0}^\infty \psi_{v,u}^*(y;T^*(F))d G^*(y)}{(\lambda^*_F)'(\cg{T^*(F)})}.
\EQNY

In Figure \ref{fig.3}, we take $G(x)=\Gamma(x; \lambda, \beta)$ with scale parameter $\lambda = 0.5$ and  shape parameter $\beta\in(0,5)$, {which is a Weibull-type distribution with $\alpha = 1$.} Its density function $g(x; \lambda, \beta)$ is given by
\BQN
\label{d Gamma}
g(x; \lambda, \beta) = \frac{\lambda^{\beta}}{\Gamma(\beta)}x^{\beta-1}\exp\{-\lambda x\}, \quad x \cL> 0.
\EQN
We see that, \cg{the absolute values of the influence functions} of both $M$-estimations $\tilde T_n$ and $T_n^*$ \cg{are} increasing in $\beta$, and decreasing with $v$. In other words,  with increasing huberization and light-tail contamination, one \cg{gets} the reduction of sensitivity to deviations from the Weibull model.

\begin{figure}[H]
	\centering
	\setlength{\abovecaptionskip}{0pt}
	\includegraphics[width=16cm,height=6cm]{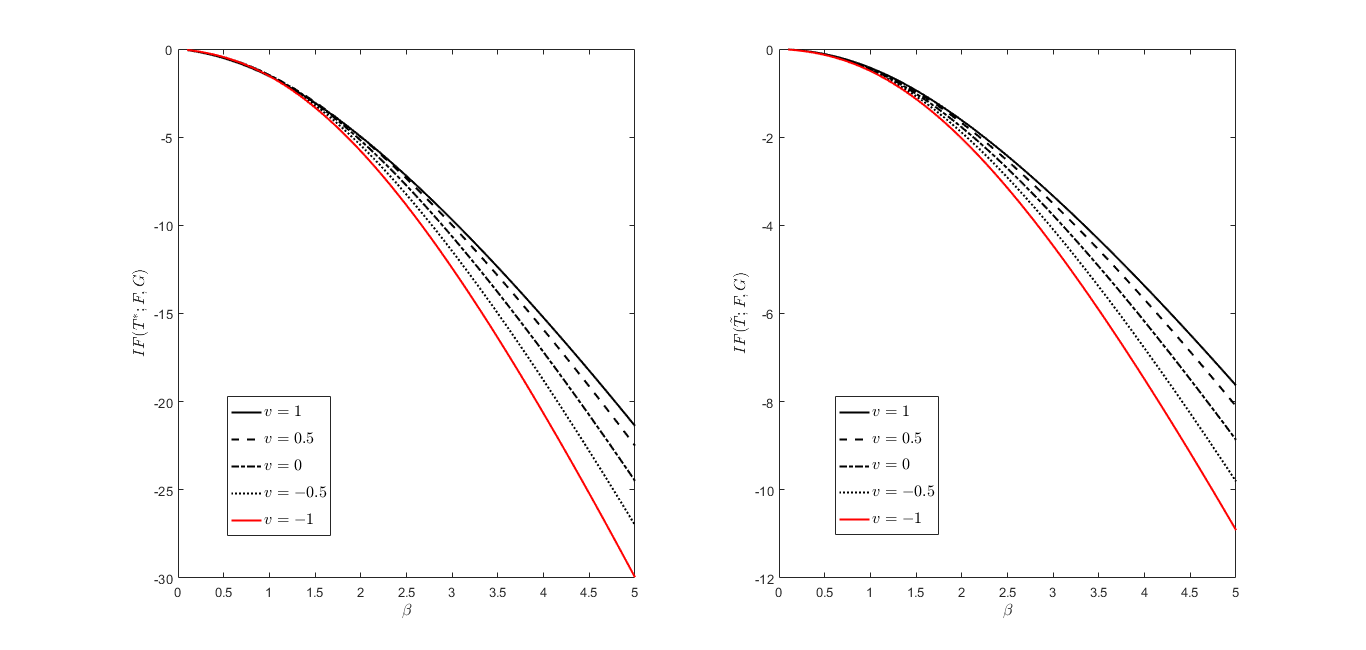}\\
	\caption{{Influence functions $IF(T;F,G)$ for $T= \tilde T^{(v, \infty)}_n$ (\textbf{left}) and $T^{*(v, \infty)}_n$ (\textbf{right}).~Here $G(x)=\Gamma(x; \lambda,\beta),\,$ $\lambda=0.5$, $\beta\in(0,5), v= \pm1, \pm 0.5, 0$ and $F(x)= F_W(x; \alpha_0)$ is given by \eqref{def_Weibull} with $c_0=1, {\alpha_0=1}$.}}\label{fig.3}
\end{figure}

\section{Simulations}
\label{Section 4}

In this section, we carry out a simulation study to illustrate the small sample {behavior} of $M$-estimations $\tilde T_n^{(v,\IF)}$ and $T_n^{*(v,\IF)}$ compared to the maximum likelihood estimation $\hat\alpha_{mle} = \tilde T_n^{(-1,\IF)}$ and the classical Hill-type estimation $\hat\alpha^{(k_n)}_{Hill}$ of the Weibull tail coefficient given by {(cf. \cite{girard2004hill})}
\BQN
\label{d Hill}
\hat{\alpha}_{Hill}^{(k_n)}=\frac{\frac{1}{k}\Sigma_{j=1}^k \log(\log\frac{n+1}{j})-\log(\log\frac{n+1}{k+1})} {\frac{1}{k}\Sigma_{j=1}^k \log(X_{n-j+1,n})-\log(X_{n-k,n})}, \quad k=1, \ldots, n-1.
\EQN

To analyze the robustness of the $M$-estimations, we generate $m=1000$ samples of size $n = 30, 50, 80$ and 100 from
Weibull distribution $F_W(x; \alpha) = 1-\exp\{-c_0 x^\alpha\}, x>0$ contaminated by Gamma distribution $\Gamma(x; \lambda, \beta)$ with contamination level $\epsilon\in(0,1)$, i.e., the underlying risk follows
\BQNY
F_{\epsilon}(x) = (1 - \epsilon)F_W(x; \alpha) + \epsilon \Gamma(x; \lambda, \beta).
\EQNY

In the simulations, we take $c_0 = 0.5, 1,2, \cg{d_0=1, d_1=2, \alpha=1,2}$ and $\lambda = \beta=0.5, \epsilon = 0.1, 0.3$. {Table \ref{tabel.1} lists}  the average estimations $\bar \alpha$, the sample variance  $s^2$ and the ratio of mean squared error (MSE) of MLE, Hill-type estimation to that of $\tilde T_n$ and $T_n^*$ with $v=0$, i.e.,  \cg{$(\hat r, \tilde r, r^*)= (\hat r_0, \tilde r_0, r_0^*)$} is given by
\cg{\BQN
	\label{def_ratio}
	\hat {r}_{v}= \frac{MSE(\hat{\alpha}^{(k_{opt})}_{Hill})}{MSE(\tilde{T}_n^{(v,\infty)})},
	\quad \tilde{r}_{v}=\frac{MSE(\hat\alpha_{mle})}{MSE(\tilde{T}_n^{(v,\infty)})},
	\quad r^*_{v}=\frac{MSE(\hat\alpha_{mle})}{MSE(T^{*(v,\infty)}_n)}.
	\EQN}

Here, we use alternatively $k_n =k_{opt}$ given by (since the traditional \cL{optimal} choice of $k_n$ in \cite{girard2004hill} is not available for small samples)
\BQNY
k_{opt} = \arg\min_{k_n\ge1} MSE(\hat\alpha_{Hill}^{(k_n)}).
\EQNY

{The last column of Table \ref{tabel.1} is} the relative proportion of $k_n$ for which $MSE(\hat{\alpha}_{Hill}^{(k_n)})\leq MSE(\tilde{T}_n^{(v,\infty)})$, denoted by $p_{Hill}$,  is given by
\BQNY
p_{Hill} = \frac{\#\{1\le k_n\le n-1: MSE(\hat{\alpha}_{Hill}^{(k_n)})\leq MSE(\tilde{T}_n^{(v,\infty)}) \}}{n-1}\times 100\%.
\EQNY

The $p_{Hill}$ {describes} {the percent} that the Hill-type estimation outperforms the estimation $\tilde{T}_n^{(v,\infty)}$.
%In  Table \ref{tabel.1}, we take the contamination level $\epsilon = 0.1, 0.3$, the Weibull distribution $F_W(x; \alpha)$ with parameter $c_0= 0.5, 1, 2$ and $\alpha=1,2$, and  Gamma contamination $\Gamma(0.5, 0.5)$.

\begin{table}[H]
	\centering
	\setlength{\abovecaptionskip}{1pt}
	\makeatletter\def\@captype{table}\makeatother\caption{Comparisons of $\tilde T_n, \, T^*_n$ with $\hat\alpha_{mle},\, \hat{\alpha}_{Hill}^{(k_n)}$. Here we take $m=1000$ samples of size $n=30$, 50, 80, 100 from $F_{\epsilon}(x)=(1-\epsilon)F_W(x; \alpha)+\epsilon \Gamma(x; 0.5, 0.5)$.}\label{tabel.1}
	\resizebox{\textwidth}{!}{\begin{tabular}{crccccccccrrrl}
			\toprule
			\multicolumn{1}{c}{\boldmath{$(\epsilon, c_0, \alpha) $}} &\boldmath{$n$} & \boldmath{$\bar{\alpha}_{mle}$} &\boldmath{$\bar{\alpha}_{Hill}$} &\boldmath{$\tilde{T}_n$} &\boldmath{$T_n^*$} &\boldmath{$s_{mle}^{2}$} &\boldmath{ $s_{Hill}^{2} $}  &\boldmath{$s_{\tilde{T}}^{2}$} &\boldmath{$s_{T^*}^{2}$} &\boldmath{$\hat{r}$} &\boldmath{$\tilde{r}$} & \boldmath{$r$}\textbf{*}  &\boldmath{$p_{Hill}$}\\ \midrule
			\multirow{4}*{(0.3, 1, 1)} &\multicolumn{1}{r}{30} &0.9217  &0.8255  &1.0006  &1.0005  &0.0072  &0.0451  &0.0018  &0.0015  &68.6811  &8.7774  &8.3678  &0.00\\
			~ &\multicolumn{1}{r}{50} &0.8609  &0.8435  &1.0022  &1.0061  &0.0068  &0.0289  &0.0014  &0.0015  &23.1288  &14.3407  &13.4344  &0.00\\
			~ &\multicolumn{1}{r}{80} &0.8274  &0.8326  &1.0083  &1.0075  &0.0041  &0.0176  &0.0013  &0.0016  &9.7485  &19.5972  &19.2393  &0.00\\
			~ &\multicolumn{1}{r}{100} &0.8161  &0.8368  &1.0147  &1.0119  &0.0030  &0.0148  &0.0018  &0.0015  &6.2607  &20.8143  &20.0842  &0.00\\\midrule
			\multirow{4}*{(0.1, 1, 1)} &\multicolumn{1}{r}{30} &0.9885  &0.9343  &0.9942  &0.9940  &0.0007  &0.0469  &0.0007  &0.0006  &5.7287  &1.2283  &1.0561  &0.00\\
			~ &\multicolumn{1}{r}{50} &0.9834  &0.9252  &0.9949  &0.9952  &0.0009  &0.0269  &0.0007  &0.0006  &3.6407  &1.8194  &1.8000  &0.00\\
			~ &\multicolumn{1}{r}{80} &0.9776  &0.9407  &0.9962  &0.9953  &0.0009  &0.0189  &0.0005  &0.0006  &1.0006  &2.5849  &2.3829  &0.00\\
			~ &\multicolumn{1}{r}{100} &0.9735  &0.9302  &0.9964  &0.9961  &0.0010  &0.0130  &0.0005  &0.0005  &0.7138  &3.2549  &2.9623  &0.15 \\\midrule
			\multirow{4}*{(0.3, 1, 2)} &\multicolumn{1}{r}{30} &1.2382  &1.6039  &1.9960  &1.9919  &0.0687  &0.2408  &0.0056  &0.0050  &74.7362  &147.2932  &126.8653  &0.00\\
			~ &\multicolumn{1}{r}{50} &1.1347  &1.6443  &2.0015  &1.9963  &0.0268  &0.1576  &0.0045  &0.0042  &22.7834  &158.1670  &165.1127  &0.00\\
			~ &\multicolumn{1}{r}{80} &1.0851  &1.6853  &2.0050  &2.0039  &0.0127  &0.1219  &0.0039  &0.0038  &7.8035  &197.7217  &180.3746  &0.00\\
			~ &\multicolumn{1}{r}{100} &1.0731  &1.6709  &2.0081  &2.0073  &0.0085  &0.0883  &0.0042  &0.0038  &5.1102  &223.2889  &186.3996  &0.00\\
			\midrule
			\multirow{4}*{(0.1, 1, 2)} &\multicolumn{1}{r}{30} &1.9245  &1.8399  &1.9903  &1.9859  &0.0169  &0.2025  &0.0026  &0.0024  &4.8833  &9.3566  &8.4148  &0.00\\
			~ &\multicolumn{1}{r}{50} &1.8459  &1.8506  &1.9900  &1.9888  &0.0239  &0.1223  &0.0022  &0.0021  &3.1233  &21.4576  &18.4649  &0.00\\
			~ &\multicolumn{1}{r}{80} &1.7654  &1.8392  &1.9873  &1.9895  &0.0228  &0.0754  &0.0017  &0.0017  &3.0758  &39.6547  &35.4584  &0.00\\
			~ &\multicolumn{1}{r}{100} &1.7249  &1.8681  &1.9898  &1.9906  &0.0190  &0.0660  &0.0018  &0.0018  &1.9142  &43.1656  &45.1557  &0.00\\\midrule
			\multirow{4}*{(0.3, 2, 1)} &\multicolumn{1}{r}{30} &0.9466  &0.8640  &0.9974  &0.9987  &0.0047  &0.0429  &0.0016  &0.0021  &82.5558  &4.6014  &3.5335  &0.00\\
			~ &\multicolumn{1}{r}{50} &0.9061  &0.8881  &0.9955  &0.9965  &0.0051  &0.0298  &0.0015  &0.0013  &22.7834  &12.7013  &10.1286  &0.00\\
			~ &\multicolumn{1}{r}{80} &0.8729  &0.8848  &0.9944  &0.9955  &0.0036  &0.0172  &0.0012  &0.0010  &3.2990  &16.1109  &16.0286  &0.00\\
			~ &\multicolumn{1}{r}{100} &0.8562  &0.8938  &0.9970  &0.9978  &0.0029  &0.0152  &0.0011  &0.0011  &1.7146  &19.8407  &18.1323  &0.00\\\midrule
			%\multicolumn{1}{c|}{$(\epsilon, c_0, \alpha) $} &\multicolumn{1}{|r}{\quad} &\quad &\quad &\quad &\quad &\quad &\quad &\quad &\quad &\quad &\quad &\quad &\quad\\ \midrule
			\multirow{4}*{(0.1, 2, 1)} &\multicolumn{1}{r}{30} &0.9880  &0.9223  &0.9953  &0.9956  &0.0005  &0.2773  &0.0483  &0.0011  &5.1904  &0.8848  &0.7232  &0.00\\
			~ &\multicolumn{1}{r}{50} &0.9852  &0.9557  &0.9941  &0.9952  &0.0007  &0.1438  &0.0261  &0.0006  &3.4681  &1.2380  &1.1071  &0.00\\
			~ &\multicolumn{1}{r}{80} &0.9808  &0.9452  &0.9940  &0.9942  &0.0008  &0.1775  &0.0165  &0.0005  &0.9524  &2.0353  &1.6522  &0.10\\
			~ &\multicolumn{1}{r}{100} &0.9762  &0.9560  &0.9927  &0.9939  &0.0009  &0.0444  &0.0148  &0.0006  &0.8589  &2.3773  &2.2494  &0.12\\\midrule
			%$(\epsilon, c_0, \alpha) $ &\multicolumn{1}{|r}{\quad} &\quad &\quad &\quad &\quad &\quad &\quad &\quad &\quad &\quad &\quad &\quad &\quad\\ \midrule
			\multirow{4}*{(0.3, 2, 2)} &\multicolumn{1}{r}{30} &1.3019  &1.6018  &1.9855  &1.9853  &0.0583  &0.2434  &0.0030  &0.0026  &85.6893  &153.3920  &159.4866  &0.00\\
			~ &\multicolumn{1}{r}{50} &1.2012  &1.6589  &1.9850  &1.9850  &0.0228  &0.1209  &0.0024  &0.0021  &14.7957  &260.7202  &270.0256  &0.00\\
			~ &\multicolumn{1}{r}{80} &1.1570  &1.6701  &1.9844  &1.9838  &0.0099  &1.1006  &0.0018  &0.0017  &2.5650  &348.1278  &345.6587  &0.00\\
			~ &\multicolumn{1}{r}{100} &1.1484  &1.6771  &1.9832  &1.9830  &0.0069  &0.0712  &0.0017  &0.0017  &1.4158  &386.5744  &370.8852  &0.00\\\midrule
			\multirow{4}*{(0.1, 2, 2)} &\multicolumn{1}{r}{30} &1.9238  &1.8054  &1.9885  &1.9870  &0.0161  &0.1763  &0.0027  &0.0032  &4.3161  &5.8439  &6.3442  &0.00\\
			~ &\multicolumn{1}{r}{50} &1.8519  &1.8637  &1.9886  &1.9850  &0.0212  &0.1201  &0.0031  &0.0023  &3.9388  &17.6548  &16.9291  &0.00\\
			~ &\multicolumn{1}{r}{80} &1.7646  &1.8565  &1.9869  &1.9849  &0.0200  &0.0696  &0.0017  &0.0019  &1.2588  &37.7912  &36.6744  &0.00\\
			~ &\multicolumn{1}{r}{100} &1.7402  &1.8839  &1.9849  &1.9843  &0.0181  &0.0756  &0.0017  &0.0017  &0.9791  &47.8601  &44.6477  &0.05\\\midrule
			\multirow{4}*{(0.3, 0.5, 1)} &\multicolumn{1}{r}{30} &0.9320  &0.7989  &1.0056  &1.0065  &0.0048  &0.0565  &0.0012  &0.0014  &7.2497  &7.7231  &6.7977  &0.00\\
			~ &\multicolumn{1}{r}{50} &0.8912  &0.8250  &1.0116  &1.0096  &0.0051  &0.0468  &0.0012  &0.0013  &8.8465  &12.2309  &10.2765  &0.00\\
			~ &\multicolumn{1}{r}{80} &0.8565  &0.8130  &1.0185  &1.0188  &0.0034  &0.0261  &0.0013  &0.0012  &2.8355  &15.4294  &14.8902  &0.00\\
			~ &\multicolumn{1}{r}{100} &0.8463  &0.8368  &1.0218  &1.0232  &0.0024  &0.0252  &0.0011  &0.0012  &1.3175  &17.0045  &16.5285  &0.00\\\midrule
			\multirow{4}*{(0.1, 0.5, 1)} &\multicolumn{1}{r}{30} &0.9874  &0.8848  &0.9968  &0.9943  &0.0005  &0.0428  &0.0006  &0.0006  &5.3788  &1.2157  &1.0236  &0.00\\
			~ &\multicolumn{1}{r}{50} &0.9853  &0.9136  &0.9972  &0.9952  &0.0005  &0.0295  &0.0005  &0.0005  &3.8457  &1.5111  &1.4974  &0.00\\
			~ &\multicolumn{1}{r}{80} &0.9799  &0.9193  &0.9977  &0.9975  &0.0006  &0.0181  &0.0004  &0.0005  &1.9436  &2.3528  &1.9708  &0.00\\
			~ &\multicolumn{1}{r}{100} &0.9783  &0.9165  &0.9991  &0.9989  &0.0006  &0.0143  &0.0005  &0.0004  &0.9241  &2.2865  &2.1840  &0.10\\\midrule
			\multirow{4}*{(0.3, 0.5, 2)} &\multicolumn{1}{r}{30} &1.3277  &1.5964  &2.0065  &1.8144  &0.0713  &0.2504  &0.0052  &0.0004  &61.5168  &111.5918  &15.1011  &0.00\\
			~ &\multicolumn{1}{r}{50} &1.2141  &1.6243  &2.0185  &1.8083  &0.0373  &0.1607  &0.0049  &0.0003  &31.3754  &129.6489  &17.7850  &0.00\\
			~ &\multicolumn{1}{r}{80} &1.1618  &1.6596  &2.0357  &1.8042  &0.0147  &0.1047  &0.0046  &0.0002  &13.2211  &128.0530  &18.8915  &0.00\\
			~ &\multicolumn{1}{r}{100} &1.1486  &1.6707  &2.0386  &1.8035  &0.0125  &0.0974  &0.0047  &0.0002  &4.5564  &118.1708  &19.1617  &0.00\\\midrule
			%$(\epsilon, c_0, \alpha) $ &\multicolumn{1}{|r}{\quad} &\quad &\quad &\quad &\quad &\quad &\quad &\quad &\quad &\quad &\quad &\quad &\quad\\ \midrule
			\multirow{4}*{(0.1, 0.5, 2)} &\multicolumn{1}{r}{30} &1.9443  &1.8589  &1.9900  &1.8040  &0.0093  &0.2091  &0.0020  &0.0005  &8.5329  &6.7537  &0.3454  &0.00\\
			~ &\multicolumn{1}{r}{50} &1.8936  &1.8745  &1.9935  &1.7963  &0.1316  &0.6520  &0.0020  &0.0003  &4.9060  &12.6372  &0.5832  &0.00\\
			~ &\multicolumn{1}{r}{80} &1.8329  &1.8271  &1.9958  &1.7937  &0.0720  &0.6408  &0.0018  &0.0002  &1.3524  &22.1326  &0.9232  &0.00\\
			~ &\multicolumn{1}{r}{100} &1.8125  &1.8538  &2.0024  &1.7930  &0.0670  &2.8163  &0.0017  &0.0002  &0.9561  &26.7188  &1.0363  &0.08\\
			\bottomrule
		\end{tabular}
	}
\end{table}
\newpage
We conclude from Table \ref{tabel.1} that

(i) The bias of the proposed $M$-estimations is smaller than that of Hill-type estimation and MLE estimation (see columns 2--5 for details).
	
(i\!i) The sample variance $s^2$ of our estimations is very close to zero. Note by passing that \cg{even with the \cL{optimal} choice of $k_n = k_{opt}$}, the $s^2$ of Hill-type \cg{estimations} is still relatively larger than the other (see columns 6--9 for details).
	
(i\!i\!i) Since the ratios of MSE satisfy
	$\hat r \le \cL{r^*} \le \tilde r$, we see \cg{that} the best rank estimation is $\tilde T_n$, which~coincides with the analysis of the relative  efficiency (see columns 10--12 and Figure \ref{fig.2}).
	
(i\!i\!i\!i) outperforms Hill-type estimators $\hat{\alpha}_{Hill}^{(k_n)}$ for almost all $k_n$'s. For $n = 80$, $p_{Hill}$ does not exceed $10\%$ in most cases which means that there is a set $K$ with at most $s$ = 8 of $k_n\in K$ such that the Hill-type estimators would outperform $\tilde{T}^{(0,\infty)}_n$. Similar argument holds for $n=100$. Hence, the $M$-estimations \cg{perform} better even for small samples.
	
\section{Empirical Study}
\label{Empirical}
The CRIX, a market index (benchmark), is designed by \cite{trimborn2016CRIX}. It enables each interested party to study the performance of the crypto market as a whole or single crypto market, and therefore attracts increasing attention of risk managers and regulators.
We select the daily CRIX index during 31 July 2014--1 January 2018 
%%NOTE: 20140731-20180101 has been changed, please confirm.
(available on crix.berlin) and take all $n=713$  positive log returns of CRIX multiplied by 15 {to obtain a moderate amount of sample of size $m$ around 35--50 greater than 1 for the $M$-estimation $\tilde T_n$ (recall scaled risks keep the same tail decay feature)} as the original data sequence $X=(X_i, i=1, \ldots, n)$.

In \cg{Figure} \ref{fig.4} we employ the empirical mean excess function from extreme value theory to analyze its tail feature (set below $\mathbb I\{\cdot\}$ as the indicator function)
\BQNY
\widehat m_X(t) = \frac{\sum_{i=1}^n\left(X_i -t\right)\mathbb I\{X_i>t\}}{\sum_{i=1}^n\mathbb I\{X_i>t\}} \quad \mbox{for large} \ t,
\EQNY
where $X_i$'s \cg{are} the scaled daily log returns of CRIX. We see that the log mean excess \cg{function behaves linearly} for large threshold, indicating the Weibull tail feature of the data-set (cf. \mbox{\cite{dierckx2009new}}).

\begin{figure}[H]
	\centering
	\setlength{\abovecaptionskip}{0pt}
	\includegraphics[width=12cm,height=6cm]{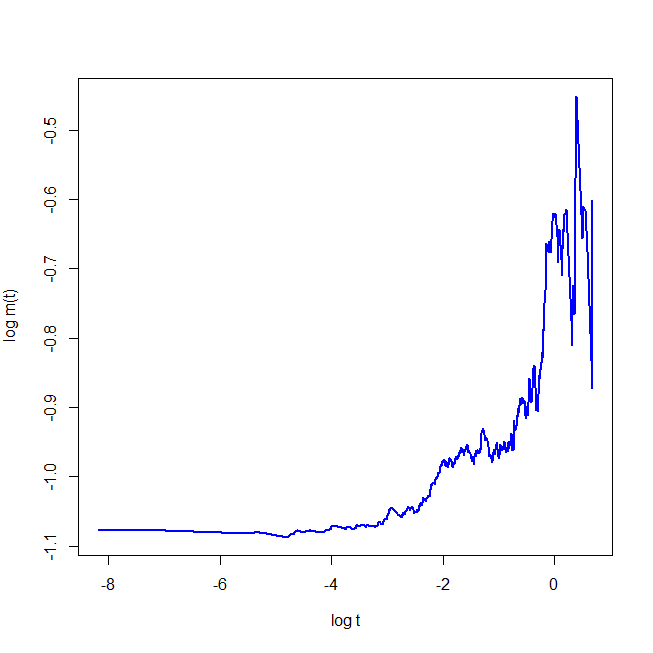}\\
	\caption{Graph of log mean excess function of scaled log returns of daily CRIX during 31 July 2014--1 January 2018.}\label{fig.4}
	%%NOTE: The figure seems a little stretched, please check.
\end{figure}

Therefore, we illustrate the robustness of the proposed $M$-estimations $\tilde{T}^{(v,\infty)}_n$ and $T^{*(v,\infty)}_n$ with \cg{$(d_0, d_1) = (0.8848, 0.9898)$ as the $95\%$ confidence interval via MLE,} and $v=0$ using the real data-set $X$ and compare it with the Hill-type estimations $\hat\alpha_{Hill}^{(k_n)}$ given by \eqref{d Hill}. Specifically, we~consider the same contamination distribution $G(x)= \Gamma(x; 0.5, 0.5)$ and contamination level $\epsilon = 0.05i$, \mbox{$i=0,1,\ldots, 10$}. Besides, the sample fraction $k_n$ involved in the Hill-type estimations, is chosen via the bootstrap and maximum likelihood method as follows.
\BQN
\label{def_k}
k_{opt}^{(1)} = \arg\min_{k_n\ge1} |\hat\alpha_{Hill}^{(k_n)} - \bar{\alpha}_{b-Hill}^{(k_n)}|, \quad k_{opt}^{(2)} = \arg\min_{k_n\ge1} |\hat\alpha_{Hill}^{(k_n)} - \hat{\alpha}_{mle}|,
\EQN
where $\bar{\alpha}_{b-Hill}^{(k_n)}$ is the average \cg{of} Hill-type \cg{estimations} based on $m=100$ bootstrap samples, and $\hat{\alpha}_{mle}$ is the maximum likelihood estimation of the shape parameter  $\alpha$ of Weibull distribution (see \eqref{def_Weibull} for its definition).
Due to the unknown Weibull tail coefficient $\alpha$, we use alternatively the relative deviation of $\hat\alpha$ at contamination level $\epsilon$ to $\epsilon+\delta_{\epsilon}$, denoted by $D(\hat\alpha)$ to study the relative robustness. Specifically,
\BQN
D(\hat\alpha)=Deviation(\hat\alpha)=|\hat\alpha(\epsilon+\delta_{\epsilon})-\hat\alpha(\epsilon)|,
\EQN
where $\hat\alpha={\tilde{T}_n, T^*_n}, \hat\alpha_{Hill}^{(1)}$ \cg{and} $\hat\alpha_{Hill}^{(2)}$ stand for the  $M$-estimations and Hill-type estimations with {optimal} choice of $k_n$ as in \eqref{def_k}, {accordingly}.

\cg{From Table \ref{table.4}, we draw the following conclusions:} (i) As expected, the proposed $M$-estimations are not sensitive to the contaminations, since the relative \cg{deviations} of $M$-estimations are almost zero. Conversely, both Hill-type estimations with optimal choices of sample fraction have obvious deviations from no contamination to small contamination \cL{($D(\hat\alpha_{Hill}^{(1)}) = 0.1277, D(\hat\alpha_{Hill}^{(2)}) = 0.2601$ for $\epsilon = 0$)}. (ii)~The  Hill-type estimation \cg{$\hat\alpha_{Hill}^{(2)}$}, with average value around \cg{0.67}, underestimates the $\alpha$ to some extent since the averages of the other three estimations are closer to 0.80.

\begin{table}[H]
	\centering
	\setlength{\abovecaptionskip}{1.5pt}
	\caption{Estimations of Weibull tail coefficient and its relative deviations via contamination level \mbox{$\epsilon = 0.05 i$}, \mbox{$i=0,\ldots, 10$}. Data is the positive and scaled log returns of daily CRIX during 31 July 2014--1 January 2018.}\label{table.4}
	\begin{tabular}{ccccccccc}
		\toprule
		\boldmath{$\epsilon$} &\boldmath{$\tilde{T}_n$} & \boldmath{$T^*_n$} & \boldmath{$\hat\alpha_{Hill}^{(1)}$}
		&\boldmath{$\hat\alpha_{Hill}^{(2)}$} &\boldmath{$D(\tilde{T}_n)$} &\boldmath{$D(T^*_n)$} &\boldmath{$D(\hat\alpha_{Hill}^{(1)})$}  &\boldmath{$D(\hat\alpha_{Hill}^{(2)})$} \\ \midrule
		0.00    &0.7711 &0.7932 &0.9202 &0.9359 &0.0072 &0.0055 &\textbf{0.1277} 
		%%NOTE: Please explain numbers in bold.
		&\textbf{0.2601}\\
		0.05 &0.7783 &0.7987 &0.7925 &0.6758 &0.0056 &0.0060 &\textbf{0.0084} &\textbf{0.0246}\\
		0.10  &0.7839 &0.8047 &0.8009 &0.6512 &0.0002 &0.0005 &\textbf{0.0038} &\textbf{0.0028}\\
		0.15 &0.7841 &0.8052 &0.8047 &0.6484 &0.0026 &0.0144 &\textbf{0.0258} &\textbf{0.0172}\\
		0.20  &0.7867 &0.8196 &0.8305 &0.6312 &0.0093 &0.0117 &\textbf{0.0560} &0.0094\\
		0.25 &0.7960 &0.8313 &0.7745 &0.6406 &0.0046 &0.0186 &0.0168 &0.0075\\
		0.30  &0.8006 &0.8499 &0.7577 &0.6331 &0.0046 &0.0092 &0.0038 &0.0089\\
		0.35 &0.7960 &0.8407 &0.7539 &0.6420 &0.0120 &0.0084 &\textbf{0.0168} &0.0049\\
		0.40  &0.8080 &0.8491 &0.7707 &0.6371 &0.0008 &0.0096 &\textbf{0.0370} &0.0029\\
		0.45 &0.8072 &0.8587 &0.7337 &0.6400 &0.0069 &0.0052 &\textbf{0.0208} &\textbf{0.0096}\\
		0.50  &0.8003 &0.8639 &0.7545 &0.6304 &-      &-      &-      &-     \\
		%\midrule
		%mean &0.7920 &0.8286 &0.7903 &0.6696 &0.0054 &0.0089 &\textbf{0.0317} &\textbf{0.0348}\\
		%median &0.7960 &0.8313 &0.7745 &0.6406 &0.0051 &0.0088 &\textbf{0.0188} &\textbf{0.0092}\\
		\bottomrule
	\end{tabular}
\end{table}

\section{Proofs}
\label{Section 5}
\begin{proof} [Proof of Lemma \ref{L1}]Firstly, we show that $h^*(t), \, t\ge t_0$ is strictly increasing. Indeed, $h(t) = (c_0 t -1) \ln t -1, t>0$ is twice differentiable and
	\BQN
	\label{def_dh}
	h'(t) = {c_0(\ln t +1) -\frac 1 t},\quad h''(t) = \frac{c_0} t + \frac1{t^2} >0,
	\EQN
	which \cg{imply} that $h(t), t>0$ is a convex function with a unique minimum $h(t_0^*)$ where $h'(t_0^*) = 0$. Therefore, we have $t_0 = \arg \min_{t\ge 1}{h(t)} $ exists and the unique solution $t_0= \max(t_0^*, 1)$ and thus $h^*(t), t\in [t_0, \infty)$ is strictly increasing.
	Noting further for given $t_0\ge 1$ that $t_0^{1/\alpha}$ is strictly decreasing in $\alpha$, we have $h^*(x^\alpha)$ is strictly increasing in $[x_0, \infty)$ with $ x_0= t_0^{1/d_0} \ge t_0^{1/\alpha}$ since $\alpha \ge d_0$.
	
	Secondly, note that $1-\tilde F_W(x; \alpha) = \exp\{-c_0(x^\alpha-1)\}, \, x\ge1$. It follows by some elementary calculations that
	$-\alpha^{-1} \tilde h(x^\alpha)$ is the score function of $\tilde F_W(x; \alpha)$. Moreover, in view of \eqref{def_dh}, the minimizer $t_0^*$ of $h$ is decreasing in $c_0$. This together with the fact that $h(1) = -1$ implies that  $\tilde h^{\leftarrow}(y) = \inf\{t \ge 1: \tilde h(t) \ge y\}, y\in [-1, \infty)$ is strictly increasing.
\end{proof}

\begin{proof} [Proof of Theorem \ref{T1}]
	It follows by \eqref{def_psi1} that $\tilde\psi_{v,u}(y;\alpha)$ is strictly increasing and continuous in $\alpha$. Hence~it suffices to show that
	$$\tilde \lambda(\alpha):=\tilde\lambda_{F_W}(\alpha)=\int_{1}^\infty \tilde\psi_{v,u}(y;\alpha)d \tilde F_W(y;\alpha_0) $$
	has an isolated root $\alpha=\alpha_0$. We have
	\begin{align}
	\label{def_mean1}
	\tilde \lambda(\alpha)%&=\int_{1}^{\infty}\tilde\psi_{v,u}(y;\alpha)d \tilde F_W(y;\alpha_0) \notag \\
	&=\int_{1}^{\infty}[\tilde h(y^\alpha)]_{v}^{u}d \tilde F_W(y;\alpha_0)-\int_{1}^{\infty}[\tilde h(y^\alpha)]_{v}^{u}d \tilde F_W(y;\alpha) \notag \\
	&= \int_{1}^{\infty}[\tilde h(y^\alpha)]_{v}^{u}d \tilde F_W(y;\alpha_0)- \tilde \mu, \quad \tilde \mu:=v + \int_v^u \exp\{-c_0[\tilde h^\leftarrow(z)-1]\} \, dz.
	\end{align}
	
	Next, it  follows by a change of variable $t= \tilde h(y^\alpha)$ and \cg{integration by parts} that
	\begin{align*}
	\lefteqn{ \int_{1}^{\infty}[\tilde h(y^\alpha)]_{v}^{u}d \tilde F_W(y;\alpha_0)}\\
	&=\int_{-1}^{v} v d  \tilde F_W([\tilde h^\leftarrow(t)]^{1/\alpha};\alpha_0) +\int_v^u t d \tilde F_W([\tilde h^\leftarrow(t)]^{1/\alpha};\alpha_0)  +\int_{u}^{\infty} u d \tilde F_W([\tilde h^\leftarrow(t)]^{1/\alpha};\alpha_0) \\
	%&=v [1- F_W(x_0;\alpha_0)] + \int_{v}^{u}[1-F_W(g^{\leftarrow}(t;\alpha);\alpha_0)] dt + v F_W(x_0; \alpha_0) \\
	&= v + \int_{v}^{u} \exp\{-c_0 [(\tilde h^\leftarrow(t))^{\alpha_0 / \alpha}-1]\} dt.
	\end{align*}
	
	Hence, $\tilde \lambda(\alpha_0) = 0$ and
	\BQN\label{def_diff1}
	\tilde \lambda'(\alpha) = \frac{c_0\alpha_{0}}{\alpha^{2}}\int_{\tilde h^\leftarrow (v)}^{\tilde h^\leftarrow(u)} s^{\alpha_0 / \alpha} \ln  s \exp\{-c_0 [s ^{\alpha_{0}/\alpha}-1]\}  d \tilde h(s) >0
	\EQN
	since  $\tilde h^\leftarrow(s)$ is strictly increasing over $[1, \infty)$ and
	\BQNY
	s> \tilde h^{\leftarrow}(v)\geq \tilde h^\leftarrow(-1) = 1.
	\EQNY
	
	Consequently, the consistency of $\tilde T_{n}$ is obtained.
	
	Next, we show the asymptotic normality of $\tilde T_n$. Set below (recall $\tilde\mu$ given in \eqref{def_mean1})
	\BQNY
	\tilde{\sigma}_{v,u}^{2}(\alpha) := \int_{1}^\infty \tilde \psi_{v,u}^2(y; \alpha) d \tilde F_W(y;\alpha_0) =\int_{1}^{\infty}([\tilde h(y^\alpha)]_{v}^{u}-\tilde \mu)^{2}d \tilde F_W(y;\alpha_0).
	\EQNY
	
	Since
	\begin{align*}
	\tilde{\sigma}_{v,u}^{2}(\alpha)%&=\int_{1}%^{\infty}\psi_{v,u}^*^2(y,\alpha)dF^*(y;%\alpha_0)=\int_{1}^{\infty}([\tilde h(y^\alpha)]_{v}%^{u}-\tilde \mu)^{2}dF^*(y;\alpha_0)\\
	&=\int_{-1}^{v} (v-\tilde \mu)^{2} d \tilde F_W([\tilde h^\leftarrow(t)]^{1/\alpha};\alpha_0) +\int_{v}^{u}(t-\tilde \mu)^{2}d \tilde F_W([\tilde h^{\leftarrow}(t)]^{1/\alpha};\alpha_0) \\
	& \quad +\int_{u}^{\infty} (u-\tilde \mu)^{2} d \tilde F_W([\tilde h^{\leftarrow}(t)]^{1/\alpha};\alpha_0) \\
	%&=(u-\tilde \mu)^{2}-\int_{v}^{u} F(g^{\leftarrow}(t;%\alpha);\alpha_{0})d (t-\tilde \mu)^2\\
	&= (v-\tilde \mu)^{2}+ 2\int_{v}^{u} (t-\tilde \mu)\exp\{- c_0[(\tilde h^\leftarrow(t))^{\alpha_0 / \alpha}-1]\} d t \\
	&= (v-\tilde \mu)^{2}+ 2\int_{\tilde h^\leftarrow(v)}^{\tilde h^\leftarrow(u)} (\tilde h(s)-\tilde \mu)\exp\{- c_0[s^{\alpha_0 / \alpha}-1]\} d \tilde h(s)
	\end{align*}
	is finite in a neighbourhood of $\alpha_{0}$ and continuous at $\alpha=\alpha_{0}$. It follows thus by Theorem A, p.\,251 in \cite{serfling2009approximation} that $\tilde T_n$ is asymptotically normal distributed.
	
	Furthermore, we have by \eqref{def_diff1}
	\BQNY
	\tilde \lambda'(\alpha_0) = \frac{c_0}{\alpha_0}\int_{\tilde h^\leftarrow (v)}^{\tilde h^\leftarrow(u)} s \ln  s \exp\{-c_0 (s-1) \}  d \tilde h(s) >0.
	\EQNY
	
	Hence, the asymptotic variance of $\sqrt m (\tilde T_n - \alpha_0)$ is given \cg{by}
	$$
	\tilde \sigma_{v,u,\alpha_0;F_W}^{2} = \frac{\tilde{\sigma}_{v,u}^2(\alpha_0)}{[\tilde \lambda'(\alpha_0)]^2}.
	$$
	
	Please note that $m/n \stackrel{p}{\to} P\{X \ge 1\}= \exp\{-c_0\}$. We complete the proof of \netheo{T1}.
\end{proof}

\begin{proof} [Proof of Theorem \ref{T2}]
	Similar arguments of \netheo{T1} \cg{apply} with $\tilde \psi, \tilde F_W$ and $\tilde h$ replaced by  $\psi^*, F^*$~and $h^*$\cg{, respectively}.
	First we show the consistency of $T_{n}^*$. It follows by \eqref{def_psi} that $\psi_{v,u}^*(y;\alpha)$ is strictly increasing and continuous in $\alpha$. Hence it suffices to show that
	$$ \lambda^*(\alpha):=\lambda_{F_W}^*(\alpha)=\int_{x_0}^\infty \psi^*_{v,u}(y;\alpha)dF_W^*(y;\alpha_0) $$
	has an isolated root $\alpha=\alpha_0$. \cg{We have}
	\begin{align}
	\label{def_mean}
	\lambda^*(\alpha)&=\int_{x_0}^{\infty}\psi_{v,u}^*(y;\alpha)dF_W^*(y;\alpha_0) \notag \\
	&=\int_{x_0}^{\infty}[h^*(y^\alpha)]_{v}^{u}dF_W^*(y;\alpha_0)-\int_{x_0}^{\infty}[h^*(y^\alpha)]_{v}^{u}d F_W^*(y;\alpha)  \notag \\
	&= \int_{x_0}^{\infty}[h^*(y^\alpha)]_{v}^{u}dF_W^*(y;\alpha_0)- \mu^*, \, \mu^* := v + \int_v^u \exp\{-c_0(h^*)^\leftarrow(z)\} \, dz.
	\end{align}
	
	Next, it  follows by a change of variable $t= h^*(y^\alpha)$ and \cg{integration by parts} that
	\begin{align}
	\label{def_h^*}
	\int_{x_0}^{\infty}[h^*(y^\alpha)]_{v}^{u}dF_W^*(y;\alpha_0) &=\int_{h^*(x_0^\alpha)}^{v} v d F_W([(h^*)^\leftarrow(t)]^{1/\alpha};\alpha_0) +\int_v^u t dF_W([(h^*)^\leftarrow(t)]^{1/\alpha};\alpha_0)  \notag \\
	& \quad +\int_{u}^{\infty} u d F_W([(h^*)^\leftarrow(t)]^{1/\alpha};\alpha_0) + F_W(x_0; \alpha_0) [h^*(x_0^\alpha)]|_v^u  \notag \\
	&=v [1- F_W(x_0;\alpha_0)] + \int_{v}^{u}[1-F_W([(h^*)^\leftarrow(t)]^{1/\alpha};\alpha_0)] dt + v F_W(x_0; \alpha_0) \notag \\
	&= v + \int_{v}^{u} \exp\{-c_0 ((h^*)^\leftarrow(t))^{\alpha_0 / \alpha}\}dt,
	\end{align}
	\cg{where in the second equality we use $h^*(x_0^\alpha) \le h^*(x_0^{d_1}) =v_0 \le v$.} Hence, $\lambda^*(\alpha_0) = 0$ and
	\BQN\label{def_diff}
	(\lambda^*)'(\alpha) = \frac{c_0\alpha_{0}}{\alpha^{2}}\int_{(h^*)^\leftarrow (v)}^{(h^*)^\leftarrow(u)} s^{\alpha_0 / \alpha} \ln  s \exp\{-c_0 s ^{\alpha_{0}/\alpha}\}  d h^*(s) >0
	\EQN
	since  $h^*(s)$ is strictly increasing over $(x_0^{d_0}, \infty)$ and
	\BQNY
	s> (h^*)^{\leftarrow}(v)\geq (h^*)^\leftarrow(v_0)= \cg{x_0^{d_1} = t_0^{d_1/d_0}} \ge 1.
	\EQNY
	
	Consequently, the consistency of $T^*_{n}$ is obtained.
	
	Next, we show the asymptotic normality of $T_n^*$. Set below (recall $\mu^*$ given by \eqref{def_mean})
	\BQNY
	(\sigma^*_{v,u})^{2}(\alpha) := \int_{x_0}^\infty(\psi_{v,u}^*)^2(y; \alpha) d F_W^*(y;\alpha_0) =\int_{x_0}^{\infty}([h^*(y^\alpha)]_{v}^{u}-\mu^*)^{2}d F_W^*(y;\alpha_0).
	\EQNY
	
	Since
	\begin{align*}
	({\sigma}^*_{v,u})^{2}(\alpha)%&=\int_{x_0}%^{\infty}\psi_{v,u}^*^2(y,\alpha)dF^*(y;%\alpha_0)=\int_{x_0}^{\infty}[h^*(y^\alpha)]_{v}%^{u}-\mu^*)^{2}dF^*(y;\alpha_0)\\
	&=\int_{h^*(x_0^\alpha)}^{v} (v-\mu^*)^{2} dF_W([(h^*)^\leftarrow(t)]^{1/\alpha};\alpha_0) +\int_{v}^{u}(t-\mu^*)^{2}d F_W([(h^*)^{\leftarrow}(t)]^{1/\alpha};\alpha_0) \\
	& \quad +\int_{u}^{\infty} (u-\mu^*)^{2} dF_W([(h^*)^{\leftarrow}(t)]^{1/\alpha};\alpha_0) + F_W(x_0; \alpha_0) (v-\mu^*)^2\\
	%&=(u-\mu^*)^{2}-\int_{v}^{u} F(g^{\leftarrow}(t;%\alpha);\alpha_{0})d (t-\mu^*)^2\\
	&= (v-\mu^*)^{2}+ 2\int_{v}^{u} (t-\mu^*)\exp\{- c_0((h^*)^\leftarrow(t))^{\alpha_0 / \alpha}\} d t \\
	&= (v-\mu^*)^{2}+ 2\int_{(h^*)^\leftarrow(v)}^{(h^*)^\leftarrow(u)} (h^*(s)-\mu^*)\exp\{- c_0s^{\alpha_0 / \alpha}\} d h^*(s)
	\end{align*}
	
	is finite in a neighbourhood of $\alpha_{0}$ and continuous at $\alpha=\alpha_{0}$, it follows by Theorem A, p.\,251 in \cite{serfling2009approximation} that $T_n^*$ is asymptotically normal distributed.
	
	Furthermore, we have by \eqref{def_diff}
	\BQNY
	(\lambda^*)'(\alpha_0) = \frac{c_0}{\alpha_0}\int_{(h^*)^\leftarrow (v)}^{(h^*)^\leftarrow(u)} s \ln  s \exp\{-c_0 s \}  d h^*(s) >0.
	\EQNY
	
	Hence, the asymptotic variance is given \cg{by}
	$$
	\sigma_{v,u,\alpha_0;F}^{2} = \frac{({\sigma}_{v,u}^*(\alpha_0))^{2}}{[(\lambda^*)'(\alpha_0)]^2}.
	$$
	
	We complete the proof of \netheo{T2}.
\end{proof}

\begin{proof} [Proof of Theorem \ref{T3}] The result follows by analogous arguments as in the proof of \netheo{T1}. Since $\tilde \psi_{v, u}(x; \alpha)$ is \cg{strictly increasing and contionuous} in $\alpha$, the assumptions of \netheo{T3} are sufficient for the consistency and asymptotic normality of $\tilde T_{n}$. Using further  Lemma 7.2.1A and Theorem A (see \cg{p.\,}249 and 251 therein) by \cite{serfling2009approximation}, we complete the proof of \netheo{T3}. \end{proof}

\begin{proof} [Proof of Theorem \ref{T4}] The result follows by analogous arguments as in the proof of \netheo{T2}. Since $\psi^*_{v, u}(x; \alpha)$ is \cg{strictly increasing and continuous} in $\alpha$, the assumptions of \netheo{T4} are sufficient for the consistency and asymptotic normality of $T_{n}^*$. Using further  Lemma 7.2.1A and Theorem A (see \cg{p.\,}249 and 251 therein) by \cite{serfling2009approximation}, we complete the proof of \netheo{T4}.
\end{proof}

\textbf{Acknowledgments.} The authors would like to thank the referee for his$\setminus$her important suggestions which significantly improve
this contribution.

\vskip15pt

\bibliographystyle{plain}
\bibliography{Robust_estimations_for_the_tail_index_of_Weibull-type_distribution}

\vskip15pt

\end{document}